\documentclass[a4paper,12pt]{article}
\usepackage[utf8]{inputenc}
\usepackage[english]{babel}
\usepackage{amsfonts}
\usepackage{latexsym,amsmath,amssymb,amsthm}
\usepackage[unicode=true,backref]{hyperref}

\usepackage[all]{xy}
\input diagxy
\usepackage{tikz-cd}

\font\fr=eufm10 scaled \magstep 1 
\newtheorem{teor}{Theorem}
\newtheorem{prop}{Proposition}
\newtheorem{corol}{Corollary}

\newtheorem{definition}{Definition}
\theoremstyle{definition}

\theoremstyle{remark}

\def\beq{\begin{equation}}
\def\eeq{\end{equation}}
\def\bea{\begin{eqnarray}}
\def\eea{\end{eqnarray}}
\def\beann{\begin{eqnarray*}}
\def\eeann{\end{eqnarray*}}
\def\beasn{\begin{sneqnarray}}
\def\eeasn{\end{sneqnarray}}
\def\ben{\begin{enumerate}}
\def\een{\end{enumerate}}
\def\bit{\begin{itemize}}
\def\eit{\end{itemize}}
\def\dst{\displaystyle}

\def\derpar#1#2{\frac{\partial{#1}}{\partial{#2}}}

\def\coor#1#2#3{{#1}^{#2}, \ldots, {#1}^{#3}}
\def\moment#1#2#3{{#1}_{#2}, \ldots, {#1}_{#3}}
\def\qed{\ifvmode\removelastskip\fi
{\unskip\nobreak\hfil\penalty50\hbox{}\nobreak\hfil
\hbox{\vrule height1.2ex width1.2ex}\parfillskip=0pt
\finalhyphendemerits=0 \par\smallskip}}

\def\vf{\mbox{\fr X}}
\def\df{{\mit\Omega}}
\def\Lag{{\cal L}}

\def\d{{\rm d}}

\def\Real{\mathbb{R}}

\def\Tan{{\rm T}}

\def\inn{\mathop{i}\nolimits}
\def\Cinfty{{\rm C}^\infty}

\renewcommand{\neq}{=\hspace{-3.5mm}/\hspace{2mm}}

\title{Newtonian mechanics in a Riemannian manifold}

\author{Miguel C. Mu\~noz-Lecanda\footnote{email:   
miguel.carlos.munoz@upc.edu}
\\[1ex]
\normalsize\itshape\sffamily 
Department of Mathematics,
Universitat Polit\`ecnica de Catalunya,\\ \normalsize\itshape\sffamily 
Barcelona, Spain}

\date{\today} 

\begin{document}
\maketitle

\begin{abstract}
The work done by Isaac Newton more than three hundred years ago, continues being a path to increase our knowledge of Nature. To better understand all the ideas behind it, one of the finest ways is to generalize them to wider situations. In this report we make a review of one of these enlargements, the one that bears the mechanical systems from the elementary homogeneous three dimensional Euclidean space  to the more abstract geometry of a Riemannian manifold. 
\end{abstract}
\tableofcontents
\section{Introduction}
 
Mechanics is an ancient wisdom of humankind. It is not possible to build the Egyptian Pyramids, for example, without some more or less organised intuitive  ideas of mechanics. And they were made near five thousand years ago. Archimedes, 287--212 BC,  was one of those that best exploited mechanical ideas as an interesting instrument not only to conceive several devices but to proof geometric results. As a science, mechanics began to emerge in modern times, in XVI and XVII centuries, when Galileo, Kepler and Newton tried to explain the motion of the planets they knew in the solar system using the old intuitive mechanical and geometric ideas and the carefully  collected observations of the positions of the planets annotated along the time.

Isaac Newton, 1643-1727, was the real founder of what we know as classical mechanics or analytical mechanics, as Lagrange, 1736-1813, called it. The Newton's {\sl Principia}, \cite{Newton}, have been developed along more than three hundred years and we continue increasing their understanding and applications. Near every scientific generation has a novel approach to Newtonian mechanics and has increased the range of applications.

In fact, analytical mechanics has arrived to be a basic tool to found the description of our Universe in a systematic form and it can be said that all our modern knowledge of Nature has been developed following ideas coming at the end from Newtonian mechanics.

Geometry has always been related to mechanics. They have influenced each other in a very enhancing way along the time, using both disciplines as a source of problems and methods to state and  solve questions in the other. And more and more, new geometric techniques have been introduced in mechanics to obtain new results, or for better understanding older ones, and applications.

In particular, differential geometry is a powerful tool to clarify the deep dependencies between different expressions of the same problem, or solution, and the relations among magnitudes, due to the intrinsic formulation of the theory, that is the independence of the used coordinate system to express the equation under consideration. Indeed, it has proven to be a very adequate way to express the mechanical concepts in a short and more comprehensible form.

As far as we know, the oldest reference to the name ``Mechanics in a Riemannian manifold" is 
used by R. Hermann in \cite{HE1968}, where one chapter has precisely this title. 

But clearly there are previous authors using Riemannian techniques to tackle mechanical problems. We necessarily need to cite Eisenhart, \cite{EISEN-1928} and Synge, \cite{SYN-1926, SYN-1928}, most of these works inspired by Levi--Civita, the real pionneer. Other modern references are \cite{AM-78, Ar-89,BLOCH2015, BULE2005,CC-2005,Ol-02} where this name appears together with geometric mechanics. We will use all of them without specific citation. It is interesting to note the special effort made by some authors, specially A.D. Lewis in \cite{LEWIS, Lewis2018}, in order to justify the way going from classical to geometric mechanics and the usefulness of this last approach.

At this point, it is important to say that this survey is by no means an historical review of classical or geometric mechanics. A very nice survey on historical aspects of geometrical mechanics with an extensive bibliography is \cite{Leon-2017}. Some other historical remarks are contained in \cite{Ar-89} in the introduction of the different chapters.

The aim of this review is to develop some aspects of Newtonian mechanics in a Riemannian setting, hence without reference to the usual vector, or affine, space structure for the configuration manifold. With respect to present developments and applications we will give only some of them with names of people involved and general references in the corresponding sections, references which contain a large amount of more specific references.

With this ideas in mind, the organization of the paper is as follows:
\begin{description}
  \item[Section 2]: Notations, definitions and dynamical equation of the trajectories including the Lagrangian formulation for conservative and more general systems.
   \item[Sections 3 and 4]: We study constrained systems. First with holonomic constraints and secondly with nonholonomic ones. The corresponding d'Alembert principles and the subsequent dynamical equations are stated, both in the Riemannian form and in the Euler-Lagrange setting.
    \item[Section 5]: Is a short section for non autonomous systems.
     \item[Section 6]: Some classical subjects in Newtonian mechanics are included: Hamilton-Jacobi vector fields and geodesic fields and Hamilton-Jacobi equation in a Lagrangian setting. We finish the section with an approximation to stationary Euler equation for fluids as a Hamilton--Jacobi equation for a Newtonian system and comments and references on the relation between solutions to the Hamilton--Jacobi equation and to the associated Schr\"odinger equation for this kind of systems.
      \item[Section 7]: Comments and references on other topics in this approach and applications: symmetries and Noether's theorem and control of mechanical systems.             \item[Section 8]: Conclusions.
\end{description}

Observe that Sections 2, 3 and 4 are the general theory while sections 6 and 7 are devoted to some specific developments and applications where the Riemannian approach is significantly enhancing.

As general references on differential and Riemannian geometry we recommend \cite{Con2001,Lee2013,Lee2018}. With respect to analytical mechanics and geometric mechanics, see \cite{AM-78, Ar-89,Arovas2014, Ga-70,GPS-01,JS-98,LL-76, LM-sgam,MR-99,Ol-02, SC-71, Sch-2005, So-ssd}. A recent reference on the deep relations between geometry and physics is \cite{CIMM-2015}, not only with classical mechanics but including quantum physics.

As is usual in this approach we consider that our manifolds and mappings are of $\mathcal{C}^\infty$--class. Einstein index summation convention is also assumed.

\section{Newtonian dynamical systems}

\subsection{Definitions and dynamical equation}

From a mathematical point of view a \textbf{Newtonian mechanical system} is a triple $(Q,{\bf g},\omega)$, where
\ben
\item
$Q$ is a differentiable manifold ($\dim\, Q=n$).
\item
${\bf g}$ is a {\sl Riemannian metric} in $Q$. Then  $(Q,{\bf g})$ is a Riemannian manifold.
\item
$\omega$ is a differential 1-form in $Q$, called the \textbf{work form}.
\een

Being ${\bf g}$ a Riemannian metric, the work form $\omega\in\df^1(Q)$
is associated to a unique vector field ${\rm F}\in\mathfrak{X} (Q)$ such that
$\inn({\rm F}) {\bf g}=\omega$. We call  ${\rm F}$ the
\textbf{the field of forces} of the system. Clearly we can determine the system with the form of work or with its force field. In this case we denote the system as $(Q,{\bf g},{\rm F})$.

With this elements we have a differential equation: we look for curves,  $\gamma\colon[a,b]\subset\Real\to Q$,
solutions to the equation
\beq\label{Neweq}
\nabla_{\dot{\gamma}}\dot\gamma ={\rm F}\circ\gamma
\eeq
 where $\nabla$ is the {\sl Levi-Civita connection} associated to the Riemannian metric ${\bf g}$.  Recall that, given a Riemannian metric ${\bf g}$ on the manifold $Q$, the Levi-Civita connection $\nabla$ is the unique linear connection which is symmetrical, that is with null torsion tensor field, and Riemannian, that is $\nabla_Z({\bf g}(X,Y))= {\bf g}(\nabla_ZX,Y) + {\bf g}(X,\nabla_ZY)$, for every $X,Y,Z\in\mathfrak{X}(Q)$, or, what is the same $\nabla_Z{\bf g}$ for every $Z\in\mathfrak{X}(Q)$.
 
 Equation (\ref{Neweq}) is called the \textbf{Newton equation}, or dynamical equation, of the system. A {\bf trajectory} of the system is a curve solution to the Newton equation.

The manifold $Q$ is the \textbf{configuration space} of the system. If $\dim Q=n$ we say that the system has $n$ \textbf{degrees of freedom}. Its tangent bundle $\Tan Q$ is the   
 \textbf{phase space of coordinates--velocities} and the \textbf{phase space of coordinates--momenta} is the cotangent bundle $\Tan^*Q$. Both phase spaces are also called  {\bf state spaces}.
 
 As we are describing physical systems, it must exist {\bf observables} in order to make measures and obtain results from the state of the system. The observables of the system with configuration space $(Q,{\bf g})$ are the real algebra $\Cinfty (\Tan Q)$ of smooth functions defined on the phase space. Equivalently the algebra   $\Cinfty (\Tan^*Q)$. The value of an observable $f\in\Cinfty (\Tan Q)$ on a state $(q,v)\in\Tan Q$ is the real number $f(q,v)$.

\bigskip
\noindent{\bf Comment}: If the above definitions and equation are the Riemannian image of the Newtonian mechanics, they must contain in some sense the three Newton laws contained in his ``{\sl Principia Mathematica}", see \cite{Newton}, and it is so: 
\begin{enumerate}
  \item 
Equation (\ref{Neweq}) is no more than the classical $ma=F$, that is ``mass by acceleration is equal to force" written as 
$$
m\frac{\d}{\d t}v=F\, ,
$$ 
where the constant $m$ is contained in the covariant derivative, that is in the metric, which is usually called, as we will see in the sequel, the {\bf kinetic energy metric}. 

\item
Note that if ${\rm F}=0$, then the dynamical trajectories are the geodesic curves of the metric ${\bf g}$. This correspond to the first Newton Law ({\sl Inertia Law}\/). 

This is what, at present, is stated as the existence of \textbf{inertial systems} of coordinates: those coordinate systems where the Newton laws are fulfilled. In our differential geometry formulation not preferred coordinates are used and Inertia Law is a consequence of the dynamical equation.

\item
With respect to the \textsl{third law}, or action and reaction law, there are long and complicated discussions about its rank of application. For detailed comments see \cite{SPIVAK-2004,Spivak2010} and references therein. It seems that its appropriate domain is in the search of models for different types of forces between  bodies in contact or to state the relation of one system and the universe surrounding it. As our systems are studied as isolated ones, this law has no significant meaning in this geometric approach where every system is described and studied by its own structure given by the three defining elements $(Q,{\bf g},\omega)$ without relation to other external elements.

\end{enumerate}

\bigskip
If $(U,\varphi=(x^i))$ is a local chart in $Q$, and $\{\Gamma^k_{ij}\}$
are the corresponding Christoffel symbols of the Levi--Civita connection $\nabla$, then the dynamical equation is locally given by
$$
\ddot\gamma^k+\Gamma_{ij}^k\dot\gamma^i\dot\gamma^j={\rm F}^k\circ\gamma \ .
$$
The relation in coordinates between $\omega$ and $F$ is as follows:  if, $\omega=\omega_i\d x^i$ and \(\dst {\rm F}={\rm F}^i\derpar{}{x^i}\), then we have that
$$
\omega_i=g_{ij}{\rm F}^j \quad , \quad
{\rm F}^i=g^{ij}\omega_j \ ,
$$
where $g^{ij}$ are the components of the inverse matrix of ${\bf g}$ in this local chart (the ``inverse metric'').


It is well known that this relation between $\omega$ and $F$ is a particular case of the so called musical diffeomorphisms associated to the Riemannian metric ${\bf g}$, defined by:
\begin{eqnarray*}
 \flat:\Tan Q\to\Tan^*Q &	&\qquad   (q,v)\mapsto (q,\inn(v){\bf g}) \\  
 \vspace{2mm}\sharp:\Tan^* Q\to\Tan Q& &\qquad (q,\omega)\mapsto (q,\inn(\omega){\bf g}^{-1})\, ,
\end{eqnarray*}
and using these mappings we can write the dynamical equation in dual form: If $\gamma$ is the trajectory of the system, then $\inn(\dot\gamma){\bf g}$ is called the \textbf{linear momentum} and, as $\nabla$ is the Levi--Civita connection, we have that
$$
\nabla_{\dot{\gamma}}(\inn(\dot\gamma){\bf g}) =\inn(\nabla_{\dot{\gamma}}\dot\gamma){\bf g}=\inn({\rm F}\circ\gamma){\bf g}=\inn({\rm F}){\bf g}
\circ\gamma=\omega\circ\gamma\, ,
$$
that is: the dynamical equation (\ref{Neweq}) is equivalent to the dual form:
\begin{equation}\label{newtondual}
 \nabla_{\dot{\gamma}}(\inn(\dot\gamma){\bf g}) =\omega\circ\gamma\,.
\end{equation}

Then, as if ${\rm F}=0$ then $\omega=0$, we have that the linear momentum $\inn(\dot\gamma){\bf g}$ is conserved along the motion which is the dual statement of the inertial law.

\subsection{Euler-Lagrange equations}

Given a Riemannian manifold $(Q,{\bf g})$, we can associate a natural function defined in the tangent bundle, the so called \textbf{kinetic energy}, defined by:
$$
\begin{array}{ccccc}
K & \colon & \Tan Q & \longrightarrow & \Real  \\
& & (q,v) &\mapsto & \frac{1}{2}{\bf g}(v,v)\, ,
\end{array} 
$$
whose local expression is
$$
K(q^i,v^j)=\frac{1}{2}g_{ij}(q)v^iv^j\, .
$$

For a Newtonian system $(Q,{\bf g}, {\rm F})$, with $\omega=\inn({\rm F})g$, and using the kinetic energy $K$, we can transform the Newton equation into a new form 
which is easier to state when we know the elements defining the system.

\begin{teor}:
Let $(Q,{\bf g},\omega)$ a Newtonian mechanical system, and
$\gamma\colon [a,b]\subset\Real\to Q$ a smooth curve contained in the domain $U\subset Q$ of a chart
$(U,\varphi=(q^i))$ of $Q$. Then, $\gamma$ is a solution to the dynamical equation
(\ref{Neweq}) if and only if, it satisfies the equations 
\beq
\frac{d}{d t}\left(\derpar{K}{v^j}\circ\dot\gamma\right)-
\derpar{K}{q^j}\circ\dot\gamma=
(\omega\circ\gamma)\left(\derpar{}{q^j}\right)=\omega_j \circ\gamma
=g_{ij}F^i \circ\gamma\ ,
\label{eel}
\eeq
for every $j$, which are called \textbf{Euler-Lagrange equations of the second kind}
of the system.
\end{teor}

These equations are usually written as:
$$
\frac{d}{d t}\left(\derpar{K}{v^j}\right)-
\derpar{K}{q^j}=\omega_j 
=g_{ij}F^i \, .
$$

The proof is a direct calculus on the function $K$ in local coordinates and using the local expression of the Christoffel symbols of the Levi--Civita connection:
$$
[kl,j]=g_{ij}\Gamma^i_{kl}=
\frac{1}{2}\left(\derpar{g_{jk}}{q^l}+\derpar{g_{jl}}{q^k}-\derpar{g_{lk}}{q^j}\right) \ .
$$

Why have we changed the intrinsic dynamical equation (\ref{Neweq})  into this expression in local coordinates? These equations were obtained by Lagrange in 1788 and published in \cite{Lagrange}  and are related with variational calculus. They are easier to calculate for a particular system than the dynamical equation because you don't need to know the Christoffel symbols of the connection. In fact they have the same ``formal" theoretical expression in every coordinate system, that is, you need to apply the same rule to obtain them independently of the used coordinates, called ``generalised coordinates" by Lagrange.

\bigskip
To finish this section a comment on other kinds of forces.
In this initial section we have preferred to keep us in the case of simple forces, depending only on the position coordinates, but it is usual that the mechanical forces, or the work forms, depend not only on the position coordinates but also on the  time and the velocities. For the case of time depending forces, see Section 5 where a short introduction with references is given.

The dependency on the velocities is the case of dissipative systems or electromagnetic, Lorentz, forces. Geometrically this means that $\omega\in\df^1(Q,\tau_Q)$ and ${\rm F}\in\vf (Q,\tau_Q)$, that is they are forms, or vector fields, along the natural projection of the corresponding phase space on the configuration manifold. In this case the only change we need to do in the dynamical equations is to write $\omega\circ\dot{\gamma}$, or ${\rm F}\circ\dot{\gamma}$, instead of $\omega\circ\gamma$ or ${\rm F}\circ\gamma$, respectively. For a detailed study of general forces in mechanics in a geometric way you can see \cite{God-69, Leon-2021}.

\subsection{Conservative systems}

There is a special kind of Newtonian systems, those which are of conservative, or mechanical Lagrangian, type.

A Newtonian system $(Q,{\bf g},\omega)$ is \textbf{conservative} if the work form is exact, that is, there exists $V\in\Cinfty (Q)$ such that $\omega =-\d V$. The negative sign is a customary tradition in Physics. In this case, the function $V$ is called the \textbf{potential energy} of the system. Observe that in this case the force vector field is  ${\rm F}=-{\rm grad}\ V$.

For these systems $(Q,{\bf g},\omega=-\d V)$,
the \textbf{total energy} or \textbf{mechanical energy} of the system 
is the function $E\in\Cinfty(\Tan Q$ defined as
$$
\begin{array}{ccccc}
E & \colon & \Tan Q & \longrightarrow & \Real  \\
& & (q,v) &\mapsto & K(q,v)+(\tau^*_QV)(q,v)\,.
\end{array}
$$
We usually write $E=K+V$.

As a direct consequence of the definition we have:

\begin{teor}
\textbf{(Mechanical energy conservation)}:
Let $(Q,{\bf g},\omega=-\d V)$ be a conservative Newtonian mechanical system, then
the mechanical energy $E$
is invariant, that is constant along the trajectories of the system.
\end{teor}
({\sl Proof\/})\quad
If $\gamma\colon [a,b]\subset\Real\to Q$ is a solution to the Newton equation, then
\beq
\nabla_{\dot\gamma}\dot\gamma={\rm F}\circ\gamma \quad , \quad \inn({\rm F}){\bf g}=
\omega=-\d V \ ,
\label{eqdinsis}
\eeq
then we have that

\beann
\frac{d (E\circ\dot\gamma)}{d t}
&=&
\nabla_{\dot\gamma}(E\circ\dot\gamma)=
\nabla_{\dot\gamma}\left(\frac{1}{2}{\bf g}(\dot\gamma,\dot\gamma)+V\circ\gamma\right)
\\ 
\vspace{2mm}&=&{\bf g}(\nabla_{\dot\gamma}\dot\gamma,\dot\gamma)+\d V(\dot\gamma)=0
\ .
\eeann
\qed

In this case, the Lagrange equations have a simpler expression. Consider the function $\Lag=K-\tau_Q^*V\in\mathcal{C}^\infty(\Tan Q)$, called {\bf Lagrangian} of the system. The Euler--Lagrange equations for  $(Q,{\bf g},\omega=-\d V)$ are
$$
\frac{d}{d t}\left(\derpar{K}{v^j}\circ\dot\gamma\right)-
\derpar{K}{q^j}\circ\dot\gamma=\omega_j\circ\gamma=
(-\d V\circ\gamma)\left(\derpar{}{q^j}\right)=
-\derpar{(\tau_Q^*V)}{q^j}\circ\gamma \ ,
$$
which we can write as
\beq
\frac{d}{d t}\left(\derpar{\Lag}{v^j}\circ\dot\gamma\right)-
\derpar{\Lag}{q^j}\circ\dot\gamma=0 \ ,
\label{eel1e}
\eeq
recalling that \(\dst\derpar{(\tau_Q^*V)}{v^j}=0\). From now on we write $\Lag=K-V$ for simplicity.

The local expression of $\Lag$ is 
$$
 \Lag (q,v)=(K-V)(q,v)= \frac{1}{2}g_{ij}(q)v^iv^j-V(q)\, .
$$
These conservative systems $(Q,{\bf g},\omega=-\d V)$ are called {\bf simple mechanical systems} and the associated Lagrangians {\bf natural lagrangians}.

\section{Holonomic constrained systems. D'Alembert principle}
\protect\label{sdnlh}

A relevant topic in classical mechanics is the study of systems with holonomic or nonholonomic constraints. A {\bf constraint} is a restriction in the motion of the system. This restriction may be in the configuration space: the system is obliged to move in a particular subset of this manifold. Or it may restrict the possible velocities to be reached, to be used to move.

From our geometric approach, the first ones, called {\bf holonomic constraints}, are defined by a submanifold of the configuration space where the system must remain. The other, or {\bf nonholonomic constraints}, by a submanifold of the phase space of coordinate--velocities allowing to reach all the positions in the configuration manifold.

To write the equations of motion we need to add to our postulates some new idea related to the submanifold of constraints. There are several approaches to tackle this problem. The older and best stablished is {\bf d'Alembert principle}, which in geometric terms is specially clarifying on the motion of these kinds of systems. For an historical approach to constrained systems, in particular nonholonomic ones, see \cite{Leon-2012} and the extended bibliography contained there. As d'Alembert principle is not the only one used to obtain the equations of motion, see the previous reference for other known  principles and relations among them. For a nice discussion on d'Alembert principle in different situations with a geometric viewpoint, see \cite{MARLE-98}.

\subsection{Holonomic constraints. Holonomic d'Alembert principle}

Let $(Q,{\bf g},\omega)$ be a Newtonian mechanical system and ${\rm F}\in\vf (Q)$
its force field. Let ${\bf S}$ be a submanifold of $Q$, usually called {\bf submanifold of holonomic constraints}, and
$j\colon ${\bf S}$\hookrightarrow Q$ the natural embedding. 
We intend to describe the dynamics of the given system when it is obliged to evolve in the submanifold ${\bf S}$ of the configuration space, hence with some restriction in the coordinates the system can reach.
To force this behaviour, it is compulsory to apply a new force field ${\rm R}$, called {\bf constraint force}, which obliges the system to remain in ${\bf S}$.
In general, such force depends, not only on the position, but also on the velocity; 
then ${\rm R}\in\vf (Q,\tau_Q)$ and, moreover, we don't know it;
 in fact it is a new unknown to find.

Then we have a new dynamical equation for curves
$\gamma\colon [a,b]\subset\Real\to {\bf S}$, which is
\beq
\nabla_{\dot\gamma}\dot\gamma ={\rm F}\circ\gamma+{\rm R}\circ\dot\gamma \ .
\label{eqdinS}
\eeq

To solve this problem, we introduce the so called \textbf{d'Alembert principle}:
The constraint force ${\rm R}$ is orthogonal to the submanifold ${\bf S}$;
that is, for every $q\in {\bf S}$ and for every $ u,v\in\Tan_q{\bf S}$, we have  ${\bf g}(u,{\rm R}(q,v))=0$, supposing that  ${\rm R}$ depends on the velocities. We impose that the constraint force is orthogonal to the constraint submanifold.

How to obtain the equation of motion and the expression of the constraint force? If $g_{\bf S}=j^*g$, let $\nabla^{\bf S}$ be the Levi-Civita connection in the Riemannian manifoldt $({\bf S},g_{\bf S})$. Then, we have the following natural geometric elements and consequences:
\begin{description}
  \item[a)] For every $q\in {\bf S}$ the orthogonal decomposition $\Tan_qQ=\Tan_q{\bf S}\oplus (\Tan_q{\bf S})^\perp $ and the orthogonal projections,
  \vspace{-2mm}
 $$
\pi_{\bf S}(q)=\colon \Tan_qQ\rightarrow\Tan_q{\bf S}
\quad , \quad
\pi^\perp_{\bf S}(q)=\colon \Tan_qQ\rightarrow (\Tan_q{\bf S})^\perp \ ,
$$
\item[b)] The global orthogonal projections
\vspace{-2mm}
$$
\pi_{\bf  S}=\colon \Tan Q\vert_ {\bf S}\rightarrow\Tan {\bf S}
\quad , \quad
\pi^\perp_{\bf S}=\colon \Tan Q\vert_ {\bf S}\rightarrow \Tan {\bf S}^\perp \ .
$$
Thus d'Alembert principle is reduced to $\pi_{\bf S}\circ{\rm R}=0$.
\item[c)] The Levi--civita connection on the submanifold ${\bf S}$ satisfies $\nabla^ {\bf S}=\pi_ {\bf S}\circ\nabla$.
This can be directly proved because $\pi_ {\bf S}\circ\nabla$ is a $g_{\bf S}$--Riemannian symmetrical connection.
\end{description}

Now, taking the dynamical equation (\ref{eqdinS}) and splitting up it into the tangent and orthogonal  components with respect to ${\bf S}$, we obtain respectively
\bea
\pi_ {\bf S}(\nabla_{\dot\gamma}\dot\gamma) &=&
\pi_ {\bf S}\circ{\rm F}\circ\gamma+\pi_ {\bf S}\circ{\rm R}\circ\dot\gamma =
\pi_ {\bf S}\circ{\rm F}\circ\gamma \ ,
\label{eqdinsplit1}  \\
\pi_ {\bf S}^\perp(\nabla_{\dot\gamma}\dot\gamma) &=&
\pi_ {\bf S}^\perp\circ{\rm F}\circ\gamma+\pi_ {\bf S}^\perp\circ{\rm R}\circ\dot\gamma =
\pi_ {\bf S}^\perp\circ{\rm F}\circ\gamma +{\rm R}\circ\dot\gamma \ .
\label{eqdinsplit2}
\eea
and, denoting by ${\rm F}^{\bf S}=\pi_ {\bf S}\circ{\rm F}\in\vf ({\bf S})$ the projection of
 ${\rm F}$ on ${\bf S}$, equation \eqref{eqdinsplit1} is simply
 \beq
 \nabla^{\bf S}_{\dot\gamma}\dot\gamma = {\rm F}^{\bf S}\circ\gamma \ ;
\label{eqdinresS}
\eeq
that is: the dynamical equation of the Newtonian mechanical system
$({\bf S},g_{\bf S},\omega_{\bf S})$, where $\omega_{\bf S}=\inn({\rm F}^{\bf S})g_{\bf {\bf S}}$.

Observe that solutions to equation (\ref{eqdinresS}) are curves
$\gamma\colon [a,b]\subset\Real\to {\bf S}$ such that, introducing each of them into equation 
(\ref{eqdinsplit2}), allows us to calculate the constraint force ${\rm R}$ for that trajectory 
$\gamma$, obtaining
\begin{equation}\label{constraintR}
  \nabla_{\dot\gamma}\dot\gamma -\nabla^{\bf S}_{\dot\gamma}\dot\gamma=
{\rm F}\circ\gamma-{\rm F}^{\bf S}\circ\gamma +{\rm R}\circ\dot\gamma \ .
\end{equation}

Then we know ${\rm R}\circ\dot\gamma\in\vf (Q,\dot\gamma)$.
Notice that we can calculate the constraint force only on every trajectory of the system but {\bf not} as a vector field depending on the velocities: we need to calculate one trajectory using equation (\ref{eqdinresS}) and then we can obtain the constraint force on this trajectory by means of equation (\ref{constraintR}) where the only unknown is ${\rm R}\circ\dot\gamma$.

\bigskip
As in the case of equation (\ref{newtondual}) we can dualise equation (\ref{eqdinsplit1}) and we obtain
$$
\nabla^{\bf S}_{\dot\gamma}(\inn(\dot\gamma){\bf g}_{\bf S})=\omega_{\bf S}\circ\gamma
$$
being $\omega_{\bf S}=j^*\omega$. In fact we can state the {\bf dual d'Alembert principle} as follows: If $\rho=\inn({\rm R})g$,  then $j^*\rho =0$. And by means of the dual of the above orthogonal projections in the cotangent bundle we can obtain the value of $\rho$ on every trajectory of the system and hence of {\rm R}.

This dual form of d'Alembert principle is also called ``Principle of virtual work": the integral of the work form $\rho$ along any piece of a possible trajectory $\gamma$ of the system is zero.

\bigskip
{\bf Examples:}
\begin{enumerate}
  \item {\bf Systems with one constraint}: Let ${\bf S}=\{ q\in Q \ ;\ \varphi (q)=0\}$, with $\varphi\in\Cinfty (Q)$ 
and suppose that $\d\varphi(q)\neq 0$, for every $q\in {\bf S}$, hence ${\bf S}$ is a hypersurface of $Q$. Let $X\in\vf (Q)$ such that $\inn (X){\bf g}=\d\varphi$, then
$X$ is orthogonal to $ {\bf S}$. In this case we have
$$
\pi^\perp_ {\bf S}({\rm F})=\frac{{\bf g}({\rm F},X)}{{\bf g}(X,X)}X=
\frac{\d\varphi ({\rm F})}{\| \d\varphi\|^2}X \ ,\qquad
{\rm F}^ {\bf S}=\pi_ {\bf S}({\rm F})={\rm F}-\frac{\d\varphi ({\rm F})}{\| \d\varphi\|^2}X \ .
$$

This allows us to find the trajectories of the system as solutions to the differential equation
$$
\nabla^ {\bf S}_{\dot\gamma}\dot\gamma=
{\rm F}\circ\gamma-\frac{\d\varphi ({\rm F})}{\| \d\varphi\|^2}X \ .
$$
Once we have a trajectory solution $\gamma$, the constraint force is given by the equation
$$
\nabla_{\dot\gamma}\dot\gamma-\nabla^ {\bf S}_{\dot\gamma}\dot\gamma=
\frac{\d\varphi ({\rm F})}{\| \d\varphi\|^2}\circ\gamma-{\rm R}\circ\dot\gamma \ ,
$$
where the only unknown is ${\rm R}\circ\dot\gamma$. These expressions are related to the second fundamental form of the hypersurface ${\bf S}$.

\item {\bf Systems with several constraints}: Consider now
$ {\bf S}=\{ q\in Q \ ;\ \varphi_1(q)=0,\ldots ,\varphi_h(q)=0\}$,
with $\varphi_1\ldots ,\varphi_h\in\Cinfty (Q)$, such that
$\d\varphi_1(q),\ldots\d\varphi_h(q)$ are linearly independent at every point $q\in  {\bf S}$ (we assume that $ {\bf S}$ is not empty).
Let $\moment{Z}{1}{n-h}\in\vf (Q)$ such that:
%
\begin{description}
\item[i)] $i(Z_a)d\varphi_{\beta}=0$, for $1\leq \beta\leq h,\,\,1\leq a\leq n-h$,
\item[ii)]$ {\bf g}(Z_a, Z_b)=0, a\not= b$, $1\leq a,b\leq n-h$.
\end{description}
To obtain these vector fields $Z_a$, it is enough to take vector fields
$\moment{X}{1}{n-h}\in\vf (Q)$ satisfying the first condition (a linear equation) and apply the well known {\sl Gramm-Schmidt method}.
In this situation,  we have that
$$
\pi_ {\bf S}({\rm F})=\sum_{a=1}^{n-h}\frac{{\bf g}({\rm F},Z_a)}{{\bf g}(Z_a,Z_a)}Z_a
$$
hence, as in the previous case, we obtain the dynamical equation and the expression of the constraint force along every trajectory solution.
\end{enumerate}

The above examples can be taken as local coordinate expression for a general submanifold of the configuration manifold.

\subsection{Euler-Lagrange equations for holonomic constraints}

We have shown that for a Newtonian mechanical system $(Q,{\bf g},\omega)$
constrained to move on the submanifold $j\colon  {\bf S}\hookrightarrow Q$, the dynamics is given by the Newtonian mechanical system $( {\bf S},g_ {\bf S},\omega_ {\bf S})$.
To write the corresponding Euler-Lagrange equation of this last system, take a local chart $(U,q^i)$ in $ {\bf S}$ and the corresponding natural lifting $(\tau_Q^{-1}(U),q^i,v^i)$
to $\Tan  {\bf S}$. Then we have
\beq
\frac{d}{d t}\left(\derpar{K_ {\bf S}}{v^k}\circ\dot\gamma\right)-
\derpar{K_ {\bf S}}{q^k}\circ\dot\gamma=
(g_ {\bf S})_{ik}{(\rm F^ {\bf S})}^i \ ,
\label{equno}
\eeq
where $K_ {\bf S}\in\Cinfty(\Tan  {\bf S})$ is the {\sl kinetic energy} of the system $( {\bf S},g_ {\bf S},\omega_ {\bf S})$. It i {\bf S} easy to  show that $K_ {\bf S}=(\Tan j)^*K$.

If the dynamical system is conservative, that is $\omega=-\d V$,
then
$$
\omega_ {\bf S}=j^*\omega=-j^*\d V=-\d j^*V \ ,
$$
and the above equation takes the expression
$$
\frac{d}{d t}\left(\derpar{\Lag_ {\bf S}}{v^j}\circ\dot\gamma\right)-
\derpar{\Lag_ {\bf S}}{q^j}\circ\dot\gamma=0 \ ,
$$
where $\Lag_ {\bf S}:=(\Tan j)^*\Lag=(\Tan j)^*(K-V)$.

Notice that the constraint force is {\bf not} in these equations. This was one of the innovations developed by Lagrange, to obtain the dynamical equations without mention to the constraint force.

In fact, if $(W,x^i)$ is a local chart in $Q$, and $(U,q^j)$ is another in ${\bf S}$,
both adapted to the inclusion map $j\colon U\hookrightarrow W$, hence $U={\bf S}\cap W$, given by the local expression 
$x^i=f^i(q)$, and hence
\(\dst \dot x^i=\derpar{f^i}{q^j}\dot q^j\).
This shows that  it is enough to know the Lagrangian function $\Lag$ of the unconstrained system, to introduce these last expressions of $x^i,\dot x^i$ in the Euler-Lagrange equations of the unconstrained system and, by direct derivation, to obtain the Euler-Lagrange equations of this constrained system using the local coordinates $(q^i,\dot{q}^i)$ of $\Tan{\bf S}$, the real phase space of the system. See \cite{Ar-89} for interesting comments on this topic.

\subsection{Product systems}
Suppose we have a family of Newtonian systems, $(Q_\mu,{\bf g}_\mu, {\rm F}_\mu)$, $\mu=1,\ldots, N$. If the force fields ${\rm F}_\mu$ depend only on the corresponding configuration manifold, that is ${\rm F}_\mu\in \vf(Q_\mu)$, then we have a family of systems of differential equations  for a curve $\gamma=(\gamma_1,\ldots,\gamma_N)$, decomposed into  $N$ non--coupled equations, one for every $\gamma_\mu$. 

But if the force fields depend on the manifold product,  ${\rm F}_\mu(q_1,\ldots,q_N)$ instead of ${\rm F}_\mu(q_\mu)$, then we have a coupled family of differential equations. In some places these systems are called in interaction. We can represent the situation as another Newtonian system $(Q,{\bf g},\omega)$ with the following elements as configuration manifold, Riemannian metric, work form and force field respectively:
$$
Q=\prod_{\mu=1}^NQ_\mu , \qquad 
 {\bf g}=\oplus_{\mu=1}^N{\bf g}_\mu\,  ,\qquad\omega=(\omega_1,\ldots ,\omega_N)\,,\qquad
{\rm F}=({\rm F}_1,\ldots ,{\rm F}_N) ,
$$
where in fact $\omega_\mu\in\Omega^1(Q_\mu, \pi_\mu)$, ${\rm F}_\mu\in\vf (Q_\mu,\pi_\mu)$,  being $\pi_\mu\colon \prod_{\nu=1}^N Q_\nu\to Q_\mu$ the natural projections. 

As we have a new Newtonian system, we can consider the case of constrained one, that is a holonomic constrained system: the system is obliged to move in a submanifold ${\bf S}\subset Q$. The constraint force is also decomposed as ${\rm R}=({\rm R}_1,\ldots,{\rm R}_N)$, where  
${\rm R}_\mu\in \vf (Q_\mu,\pi_\mu)$, $\mu=1,\ldots,N$. The dynamical equation has the same form:
$$
\nabla_{\dot\gamma}\dot\gamma={\rm F}\circ\gamma +{\rm R}\circ\dot\gamma \ ;
$$
for curves $\gamma\colon [a,b]\subset\Real\to S$, $\gamma=(\gamma_1,\ldots,\gamma_N)$, $\gamma_\mu\colon [a,b]\subset\Real\to Q_\mu$, or the corresponding constraint submanifold in the case of holonomic constraints.

\section{Nonholonomic constrains. Nonholonomic d'Alembert principle}
\protect\label{slnh}

As far as we know, the description of this kind of systems is not contained in the literature with this Riemannian approach, hence we develop them in detail. For a classical approach see for example \cite{GPS-01,Som-1952}. Other geometric approaches can be seen in \cite{CLMM-2002, GMM-2003,LEWIS1998,Lewis2020}. For an extended bibliography on this topic, for both classical and geometric approaches, see \cite{Leon-2017}. In \cite{Koiller-2019}, there is a dual standpoint using Cartan equivalence that can be directly written in our Riemannian approach.

\subsection{Nonholonomic constrained systems}

Let $(Q,{\bf g},\omega)$ be a Newtonian mechanical system, ${\rm F}\in\vf (Q)$ the force field. Let $C$ be a submanifold of $\Tan Q$, $j_C\colon C\hookrightarrow \Tan Q$ the natural embedding, and suppose that $\tau_Q(C)=Q$. In this situation $C$ is called a 
 {\bf submanifold of nonholonomic constraints}.
We want to describe the dynamics of the system when it is constrained to evolve in the submanifold $C$ of the phase space. The constrained system is given by $(Q,{\bf g},F,C)$. Notice that the system is not restricted in the configuration manifold, that is in the positions, but in the possible velocities to move with.

As in the holonomic case, to solve this problem we suppose that there exists a  {\bf constraint force} ${\rm R}$, usually depending on the velocities,  that is ${\rm R}\in\vf (Q,\tau_Q)$, which forces the system to move in $C$. This constraint force is unknown.
Then the Newton dynamical equation is given for curves
$\gamma\colon [a,b]\subset\Real\to Q$ such that satisfy
\begin{enumerate}
\item
	$\dot\gamma (t)\in C$, $ t\in[a,b]$.
\item
$\nabla_{\dot\gamma}\dot\gamma ={\rm F}\circ\gamma+{\rm R}\circ\dot\gamma$.
\end{enumerate}
And we need to state conditions allowing us to find the trajectories of the system and calculate ${\rm R}$\footnote{Arnold Sommerfeld, says that this force $\mathrm{R}$ is a ``geometric force'' versus ${\rm F}$ which is an ``applied force''. See \cite{Som-1952}.}.

In order to state the nonholonomic d'Alembert principle, we need some geometric preliminaries.
Let $(q,v)\in C$. The condition assumed on $C$, $\tau_Q (C)=Q$, tells us that the dimension of the subspace of ${\rm V}_{(q,v)}(\Tan Q)$ which is tangent to $C$, does not depend on the point $(q,v)$. Let
$$
\Tan_{(q,v)}^VC={\rm V}_{(q,v)}(\Tan Q)\cap\Tan_{(q,v)}C=\{ w\in{\rm V}_{(q,v)}(\Tan Q)\ ;\ w\in\Tan_{(q,v)}C\}
$$
be the vertical subspace tangent to $C$. This is a vector subbundle of $\Tan(\Tan Q)$ and we can write $\Tan^VC={\rm V}(\Tan Q)|_{C}\cap\Tan C$ as vector bundles over the manifold $C$.

For $(q,v)\in \Tan Q$, consider the vertical lifting from the point $q\in Q$ to $(q,v)$ given, as usual, by
\begin{align*}
\lambda_q^{(q,v)}\colon&\Tan_qQ\to{\rm V}_{(q,v)}(\Tan Q)\\
&\quad u_q\,\mapsto\quad\lambda_q^{(q,v)}(u_{q}):\phi\mapsto\lim_{t\rightarrow 0}\frac{\phi(q,v+tu)-\phi(q,v)}{t},
\end{align*}
that is, the directional derivative of $\phi\in\Cinfty (\Tan Q)$ along $u_{q}$ at the point $(q,v)\in\Tan Q$.
As $\lambda_q^{(q,v)}$ is an isomorphism from $\Tan_{q}Q$ to $V_{(q,v)}(\Tan Q)$, let $(\Tan_{(q,v)}^VC)_q$ the inverse image of
$\Tan_{(q,v)}^VC\subset V_{(q,v)}(\Tan Q)$ by $\lambda_q^{(q,v)}$. Then $\Tan_{q}Q=(\Tan_{(q,v)}^VC)_q\oplus(\Tan_{(q,v)}^VC)^\perp_q$, being this one an orthogonal decomposition with respect to ${\bf g}$.

Now we introduce the \textbf{Nonholonomic d'Alembert principle}:
The constraint force ${\rm R}\in\vf (Q,\tau_Q)$ satisfies that
$$
{\rm R}(q,v)\in(\Tan^V_{(q,v)}C)^\perp_q \ ,
$$
that is, ${\bf g}({\rm R}(q,v),w)=0$, for every $w\in (\Tan^V_{(q,v)}C)_q$. The constraint force in $(q,v)\in C$ is orthogonal to the subspace $(\Tan^V_{(q,v)}C)_q\subset\Tan_q Q$ of tangent vectors at $q$ whose vertical lifting is tangent to $C$.

In the classical physics literature, the elements in $(\Tan^V_{(q,v)}C)_q$  are called {\bf virtual velocities}. 

\bigskip
\noindent{\bf Comment}: If there are no constraints, that is $C=\Tan Q$, then for every $(q,v)\in C$ we have that $\Tan_{(q,v)}^VC=V_{(q,v)}(\Tan Q)$, hence $(\Tan^V_{(q,v)}C)_q=\Tan_{q}Q$ and $(\Tan^V_{(q,v)}C)^\perp_q=\{0\}$, that is ${\rm R}(q,v)=0$ and there is no constraint force.

\bigskip
This principle allows to obtain  the expression of the constraint force and the dynamical equations of the trajectories of the system as we will see in the next paragraphs. 

\subsection{Relation between the constraint force and the constraints}

In order to do this, we need to characterize the subspace $(\Tan^V_{(q,v)}C)_q$ in relation with the \textbf{constraints}, defined as  the functions vanishing on the submanifold $C$, that is functions $\phi:\Tan Q\to\Real$  such that $j_C^*\phi =0$.

\medskip
First, let $\phi\in\Cinfty (\Tan Q)$, and consider the 1-form $\d^V\phi\in\df^1(Q,\tau_Q)$ defined by
$$
(\d^V\phi(q,v))(u)=
(\d\phi\circ\lambda_q^{(q,v)})(u)=
\d\phi (\lambda_q^{(q,v)}(u))
 ,  \quad (q,v)\in\Tan Q, \quad u\in\Tan_qQ \ ;
$$
whose expression in a local natural chart $(q^i,v^i)$ of $\Tan Q$ is \(\dst\d^V\phi=\derpar{\phi}{v^i}\d q^i\). We have the following result:

\begin{prop}\label{verticalvectors}
Let $(q,v)\in C$.
\begin{enumerate}
\item If $w\in\Tan_qQ$, then  \(\dst w\in(\Tan_{(q,v)}^VC)_q\) if, and only if,
$(\d^V\phi(q,v))(w)=0$,
for every $\phi\in\Cinfty (\Tan Q)$ such that $j_C^*\phi =0$, that is for every constraint.
\item Let $((\Tan_{(q,v)}^VC)_q)^o=\{\alpha\in\Tan^{*}_{q}Q; \alpha(w)=0, \forall w\in(\Tan_{(q,v)}^VC)_q\}\subset\Tan^{*}_{q}Q$ be the annihilator of $(\Tan_{(q,v)}^VC)_q$, then $
((\Tan_{(q,v)}^VC)_q)^o=\{\d^V\phi(q,v);\forall  \phi\in\Cinfty(\Tan Q),   j_C^*\phi=0\}.
$
\item If $w\in\Tan_qQ$, then  \(\dst w\in(\Tan_{(q,v)}^VC)_q\) if, and only if, $\inn(w){\bf g} \in((\Tan_{(q,v)}^VC)_q)^o$.
\end{enumerate}
\end{prop}
({\sl Proof\/})\quad
To prove the first item let $w\in\Tan_qQ$, then we have
\begin{eqnarray*}
&w\in(\Tan_{(q,v)}^VC)_q& \Leftrightarrow \\
&\lambda_q^{(q,v)}(w)\in\Tan_{(q,v)}^VC &\Leftrightarrow\\
\lambda_q^{(q,v)}(w)(\phi)=0, &\forall  \phi\in\Cinfty(\Tan Q),  \mathrm{with} \,\, j_C^*\phi=0
&\Leftrightarrow\\
\d\phi(\lambda_q^{(q,v)}(w))=0, &\forall  \phi\in\Cinfty(\Tan Q), \mathrm{with} \,\, j_C^*\phi=0 
&\Leftrightarrow \\
(\d^V\phi(q,v))(w)=0,&\forall  \phi\in\Cinfty(\Tan Q),  \mathrm{with} \,\, j_C^*\phi=0.&
\end{eqnarray*}
The second item is a direct consequence of the first and the third can be obtained from the definitions.
\qed
\begin{corol}
Let ${\rm R}\in\vf (Q,\tau_Q)$ and $(q,v)\in C$; then
${\rm R}(q,v)\in(\Tan^V_{(q,v)}C)^\perp_p$ if, and only if, 
$\inn({\rm R}(q,v)){\bf g}\in((\Tan_{(q,v)}^VC)_q)^o$.
\end{corol}

Usually the submanifold $C$ is given by the annihilation of a finite family of constraints, functions defined in $\Tan Q$. We are going to characterize $((\Tan_{(q,v)}^VC)_q)^o$ using these constraints. 

\bigskip
From here to the end of this section, we suppose that the submanifold $C$ is defined by the vanishing of $r$ functions $\{\phi^i\}$, with $r<n=\dim\, Q$, satisfying the condition
$$
\mathrm{rank}\,\left(\derpar{\coor{\phi}{1}{r}}{\coor{v}{1}{n}}\right)=r\, .
$$
Then $\dim\, C=2n-r$ and we have:

\begin{prop}
Let $(q,v)\in C$, then
\begin{enumerate}
\item $\dim \Tan^V_{(q,v)}C=n-r$.
\item $(\Tan^V_{(q,v)}C)_{q}=\{w\in \Tan_{q}Q; (\d^V\phi^i(q,v))(w)=0, i=1,\ldots,r \}$.
\item The subspace $((\Tan_{(q,v)}^VC)_q)^o$ is generated by $\{\d^V\phi^1(q,v),\ldots,\d^V\phi^r(q,v)\}$. Or what is the same, if $\alpha\in\Tan_{q}^{*}Q$ satisfies $\alpha |_{(\Tan^V_{(q,v)}C)_{q}}=0$, then $\alpha$ is a linear combination of $\d^V\phi^1(q,v),\ldots,\d^V\phi^r(q,v)$. 
\end{enumerate}
\end{prop}
({\sl Proof\/})\quad
Let $(q^{i}, v^{i})$ a natural coordinate system on $TQ$.
\begin{enumerate}
\item The assumed condition
\(\dst {\rm rank}\,\left(\derpar{\coor{\phi}{1}{r}}{\coor{v}{1}{n}}\right)=r\) implies that, up to a change of order in the coordinates $\coor{q}{1}{n}$, we can suppose that
$$
\det\, \left(\derpar{\coor{\phi}{1}{r}}{\coor{v}{1}{r}}\right)\not= 0.
$$
Then $(q^{1},\ldots,q^{n},\phi^1,\ldots,\phi^r,v^{r+1},\ldots,v^{n})$ is a local coordinate system of $\Tan Q$ by the Inverse Function Theorem. The vector space $V_{q,v}(\Tan Q)$ is generated by
$$
\left\{\derpar{}{\phi^1},\ldots,\derpar{}{\phi^r},\derpar{}{v^{r+1}},\ldots,\derpar{}{v^n}\right\}_{(q,v)}\, ,
$$
and the subspace  $\Tan^V_{(q,v)}C\subset V_{q,v}(\Tan Q)$ is generated by
$$
\left\{\derpar{}{v^{r+1}},\ldots,\derpar{}{v^n}\right\}_{(q,v)}\, .
$$
Hence the first item is proved.

\item The inclusion part is proved in the first item of Proposition \ref{verticalvectors} and the equality is a matter of dimensions.
\item 
The previous items imply that $\{\d^V\phi^1(q,v),\ldots,\d^V\phi^r(q,v)\}$ is a basis of $
((\Tan_{(q,v)}^VC)_q)^o$. 
\end{enumerate}
\qed
Then, as a corollary we obtain:

\begin{prop}
For $(q,v)\in C$, the form $\eta\in\df^1(Q,\tau_Q)$ satisfies \(\dst j^*_C\eta\vert_{(\Tan^V_{(q,v)}C)_q}=0\)
if and only if there exist $\moment{\lambda}{1}{r}\in\Cinfty (\Tan Q)$ such that 
$\eta=\lambda_\alpha\d^V\phi^\alpha=\lambda_1\d^V\phi^1+\ldots+\lambda_r\d^V\phi^r$
\end{prop}
({\sl Proof\/})\quad
Because for every $(q,v)\in C$, we have that $\eta(q,v)\in((\Tan_{(q,v)}^VC)_q)^o$.
\qed

\bigskip
In the case that the constraints define only locally the submanifold $C$, then the above results are valid only in the corresponding open set.

The last Proposition allows us to state the so called  
\rm{Dual d'Alembert nonholonomic principle}: 
{\it The work form $\inn({\rm R}){\bf g}$ corresponding to the constraint force ${\rm R}$ annihilates the virtual velocities of the system}.

And, as an immediate result, we have:

\begin{corol}
If ${\rm R}$ is the nonholonomic constraint force, then there exist
$\lambda_1,\ldots,\lambda_r\in\Cinfty (\Tan Q)$ such that
$$
\inn({\rm R}){\bf g}=\lambda_\alpha\d^V\phi^\alpha=\lambda_\alpha\derpar{\phi^\alpha}{v^j}\d q^j \ ,
$$
and, as a consequence,
$$
{\rm R}=\lambda_\alpha\derpar{\phi^\alpha}{v^j}g^{jk}\derpar{}{q^k} \ .
$$
\end{corol}

\begin{definition}
The functions $\lambda_1,\ldots,\lambda_r$ are called \textbf{Lagrange multipliers}
of the nonholonomic system.
\end{definition}

\noindent{\bf Comment}: Observe that in the case of holonomic constraints we could not obtain the global expression of the constraint force, we obtain the constraint form along any particular trajectory of the system, but in the present situation we can obtain one expression for ${\rm R}$ as a vector field depending on the velocities and on the Lagrange multipliers. This is because holonomic constraints are not a particular case of the nonholonomic ones. See also the comment following equation (\ref{constraintR}).

\bigskip

\noindent{\bf Important particular case}: The submanifold $C\subset\Tan Q$ is a linear subbundle of  $\Tan Q$. 
\begin{enumerate}
\item In this case $C$ is defined by the annihilation of a family of differential forms, that is, we have $\omega^{\alpha}\in\Omega^{1}(Q)$, $\alpha=1,\ldots,r$, linearly independent at every point of $Q$, and
$$
C=\{(q,v)\in\Tan Q;\, \omega^{\alpha}_{q}(v)=0,\alpha=1,\ldots,r\}\, .
$$
In local coordinates, if $\omega^{\alpha}=a^{\alpha}_{j}(q)\d q^{j}$, then $\phi^\alpha=a^{\alpha}_{j}(q)v^{j}$, that is the constraints are linear in the velocities, and the expression of the constraint force is
$$
{\rm R}=\lambda_\alpha a^\alpha_{j}g^{jk}\derpar{}{q^k} \ .
$$
\item Alternatively we can suppose that the subbundle $C$ is given as a regular distribution $\mathcal{D}$, the distribution annihilated by $\{\omega^{\alpha}, \alpha=1,\ldots,r\}$. If $(q,v)\in\mathcal{D}$, by linearity we have that  $\Tan_{(q,v)}^VC=\lambda_{q}^{(q,v)} (\mathcal{D}_{q})$, hence $(\Tan_{(q,v)}^VC)_{q}=\mathcal{D}_{q}$ and $(\Tan_{(q,v)}^VC)^\perp_q=\mathcal{D}^\perp_q$, then the constraint force ${\rm R}$ is orthogonal to $\mathcal{D}$.

This is the situation usually considered in the classical books on mechanics.

\item If the distribution $\mathcal{D}$ is integrable and $(q,v)\in\mathcal{D}$ is the initial condition of the dynamical equation for the solution $\gamma$, then the image of $\gamma$ is contained in the integral submanifold of $\mathcal{D}$ passing throught the point $q\in Q$, because $\dot\gamma(t)\in\mathcal{D}_{\gamma(t)}$ for every $t$. The constraint force ${\rm R}$, orthogonal to $\mathcal{D}$, obliges the system to move on the integral submanifolds of the constraint distribution $\mathcal{D}$.
\end{enumerate}

\bigskip
\noindent{\bf Comments}: 
\begin{enumerate}
\item We can understand the solution as follows: if we have $\phi:\Tan Q\to\Real$, an only constraint, there is an associated 1-form, $\d^V\phi$, which gives a ``constraint force'' $R^{\phi}$ such that $\inn(R^{\phi}){\bf g}$ is proportional to $\d^V\phi$.
If we have $r$ independent constraints $\{\phi^\alpha\}$, then we have the corresponding constraint forces, $R^{\phi^\alpha}$, and the subbundle generated by them, $\{R^{\phi^\alpha}\}$, and the resultant constraint force ${\rm R}$ is contained in this subbundle.
\item For these systems, d'Alembert principle says that if the system moves ``along the vertical fibres'' of C,  the work realised by the constraint force, the integral along the trajectory, is null. This is called the {\bf virtual work principle} as alternative to d'Alembert principle. 
\end{enumerate}

\subsection{Dynamical equations for nonholonomic systems}

We finish this section giving the expressions of the dynamical equations of nonholonomic constrained systems $(Q,{\bf g},\omega, C)$ in different significant cases of $C\subset \Tan Q$, the submanifold of nonholonomic constraints.
\begin{enumerate}
  \item 
If the submanifold of constraints $C$ is locally defined by the annihilation of
$r$  functions $\{\phi^\alpha\}$, and they are independent constraints,
then the dynamical equation is
$$
\nabla_{\dot\gamma}\dot\gamma ={\rm F}\circ\gamma+\lambda_\alpha\inn(\d^V\phi^\alpha)g^{-1}\circ\dot\gamma \ ,
$$
or, in dual form,
$$
\nabla_{\dot\gamma}(\inn(\dot\gamma){\bf g})=
\omega\circ\gamma+\lambda_\alpha\d^V\phi^\alpha\circ\dot\gamma \ .
$$
These equations together with the constraints defining $C$, $\phi^1=0,\ldots,\phi^r=0$,
are a system of $n+r$ equations with $n+r$ unknowns:
the components of the trajectory $\gamma$ and the Lagrange multipliers $\lambda_\alpha$. Observe that some of them are the dynamical equations and the remaining ones are the constraint functions.

The corresponding Euler-Lagrange equations are
$$
\frac{d}{d t}\left(\derpar{K}{v^j}\circ\dot\gamma\right)-
\derpar{K}{q^j}\circ\dot\gamma=
\omega_j\circ\gamma +\lambda_\alpha\derpar{\phi^\alpha}{v^j}\circ\dot\gamma
\quad , \quad (j=1,\ldots ,n)
$$
because $\omega=\omega_k\d q^k$, with $\omega_k=g_{kj}{\rm F}^j$.

\item
If the system is conservative, then $\omega=-\d V$,
where $V\in\Cinfty (Q)$ is the potential function. In this case we can introduce the Lagrangian function $\Lag=K-\tau_Q^*V$ and we have
$$
\frac{d}{d t}\left(\derpar{\Lag }{v^j}\circ\dot\gamma\right)-
\derpar{\Lag}{q^j}\circ\dot\gamma=
\lambda_\alpha\derpar{\phi^\alpha}{v^j}\circ\dot\gamma \ .
$$
As above, these equations, together with the constraints defining $C$, are also a system of $n+r$ equations with $n+r$ unknowns.

\item
If $C$ is a vector subbundle, then  $\phi^\alpha=a^{\alpha}_{j}(q)v^{j}$ and the Euler-Lagrange equations are
$$
\frac{d}{d t}\left(\derpar{K}{v^j}\circ\dot\gamma\right)-
\derpar{K}{q^j}\circ\dot\gamma=
\omega_j\circ\gamma +f_ka^k_j\circ\dot\gamma
\quad , \quad (j=1,\ldots ,n)\, .
$$
Or in the case of conservative systems
$$
\frac{d}{d t}\left(\derpar{\Lag }{v^j}\circ\dot\gamma\right)-
\derpar{\Lag}{q^j}\circ\dot\gamma=
\lambda_\alpha a^\alpha_j\circ\dot\gamma \ .
$$
\end{enumerate}
If $C$ is an affine subbundle, then  $\phi^\alpha=a^{\alpha}_{j}(q)v^{j}+b^\alpha(q)$ and the expression of the Euler-Lagrange dynamical equations is the same as above.

\section{Non-autonomous Newtonian systems}

In some interesting cases the force field acting on a Newtonian system 
depends not only on the positions and the velocities but also on time. They are called {\bf non-autonomous} or {\bf time-depending} systems.
In the following paragraphs we will try to extend the above geometric formulation to this situation. 

The geometric model appropriate to this case is the following: A \textbf{non-autonomous Newtonian mechanical system} is a triple
 $(\Real\times Q,{\bf g},{\rm F})$, where
$(Q,{\bf g})$ is a Riemannian manifold and the force field is 
${\rm F}\in\vf (Q,\pi_2)$, with $\pi_2\colon\Real\times Q\to Q$; that is,

$$
\xymatrix{&\Tan Q\ar[d]^{\tau_Q}\\{\Real\times Q}\ar[ur]^{{\rm F}}\ar[r]^{\quad\pi_2}&\, Q\,.}
$$

Moreover, if the force field depends on the velocities, then 
${\rm F}\in\vf (Q,\tau_Q\circ\rho_2)$, with $\rho_2\colon\Real\times\Tan Q\to \Tan Q$;
that is,
$$
\xymatrix{& &\Tan Q\ar[d]^{\tau_Q}\\{\Real\times \Tan Q}\ar[urr]^{{\rm F}}\ar[r]^{\quad\rho_2}&\Tan Q\ar[r]^{\tau_Q}&\, Q\,.}
$$

The Newton equations are written in the usual way:
\bit
\item
In the case that the force does not depend on the velocities
$$
\nabla_{\dot\gamma}\dot\gamma ={\rm F}\circ\bar\gamma
$$
where $\bar\gamma=(t,\gamma)\colon I\subset\Real\to I\times Q$.
We can also use the dual form by means of the corresponding work form $\omega\in\df^1(Q,\pi_2)$.
\item
If the force field depends on the velocities
$$
\nabla_{\dot\gamma}\dot\gamma ={\rm F}\circ\bar{\dot\gamma} \ ,
$$
where $\bar{\dot\gamma}=(t,\dot\gamma)\colon I\subset\Real\to I\times\Tan Q$.
As above we can use the corresponding work form $\omega\in\df^1(Q,\tau_Q\circ\rho_2)$ and obtain the equations in the dual form.
\eit

The  Euler-Lagrange equations are the same as usual but the second term depends on time $t\in\Real$.
In particular, if the work form depends on time, $\omega\in\df^1(Q,\pi_2)$,
we say that the system is \textsl{conservative} if there exists
$V\colon\Real\times Q\to \Real$, such that $\omega=-\d V_t$,
where $V_t\colon Q\to Q$ is defined by $V_t(p):=V(t,p)$,
for every $p\in Q$ and $t\in\Real$.
In this situation we can define the Lagrangian function $\Lag=K-V$, depending on time, and the Euler-Lagrange equation are as usual
$$
\frac{d}{d t}\left(\derpar{\Lag}{v^i}\circ\bar{\dot\gamma}\right)-
\derpar{\Lag}{q^i}\circ\bar{\dot\gamma}= 0 \ .
$$

The case of time depending constrained systems, both holonomic and nonholonomic, can be directly formulated with the adequate changes. The constraint force also depends on time.

It is interesting to note that if the Lagrangian function is time-depending, then the system is not conservative, that is, the energy function is not conserved along the trajectories of the system. In fact it is easy to show that
$$
\frac{dE\circ \dot\gamma}{dt}=-\frac{\partial L}{\partial t}\circ \dot\gamma\, .
$$
using the definition of the total energy from the Lagrangian, $E=K+V$.

Observe that, in the above comments we have supposed that the dependency on the time is in the forces, but there are systems where it is in the kinetic energy, that is in the Riemannian metric. This is the case of variable mass systems, like a rocket whose mass changes while the combustion goes; see \cite{Fox-67}  and \cite{Cve-16} for other interesting examples.  There are also constrained systems whose constraints depend on time, see for example \cite{Ar-89}.

Other different applications in mechanics of time depending Riemanian metrics can be seen in \cite{SarPrin-2010,MesSarCram-2011,SarPrinMes-kup-2012} and references therein. Some problems in geometry, mechanics and relativity consider the case of Riemannian metrics depending on parameters, the time for example when we have an evolution problem, but they are out of the scope of this survey.

\section{Geodesic fields, Hamilton-Jacobi equation and applications}

Hamilton--Jacobi equation is a fundamental tool to obtain significant results in the study of Lagrangian and Hamiltonian systems, but for a Newtonian system we have not, in general, this tool. Following ideas contained in \cite{CGMMR-06}, we develop in this section the associated notion, which we call Hamilton--Jacobi vector fields. They can be associated with non conservative forces, that is when we haven't a Lagrangian or Hamiltonian function. We compare both situation, the classical and the Newtonian, and give some particular applications in different fields.

\subsection{Hamilton--Jacobi vector fields}

 Consider a Newtonian system $(Q,{\bf g},{\rm F})$ and the associated dynamical equation for curves
 $\gamma:I\subseteq \mathbf{R}\longrightarrow Q$
 \begin{equation}\label{DIN}
    \nabla_{\dot{\gamma}(t)}\dot{\gamma}(t)={\rm F}(\gamma(t))\,.
\end{equation}

This is a second order differential equation for the curves $\gamma$, hence it is associated to a second order vector field in the phase space $\Tan Q$. Following ideas of Jacobi, we try to reduce the order of this differential equation and we will obtain the Hamilton--Jacobi equation for Newtonian systems in this Riemannian approach. For a similar approach in the Lagrangian setting, see \cite{CGMMR-06}. In \cite{BLMDMM-2012} you can see a more general approach in the atmosphere of skew symmetric algebroids.

First we begin with a geometric result on vector fields on the configuration space $Q$.

\begin{teor}
The vector field $X\in\mathfrak{X}(Q)$ satisfies the equation
\begin{equation}\label{HJL}
  \nabla_X X={\rm F}
  \end{equation}
if and only if its integral curves $\gamma:I\rightarrow Q$ are trajectories of the above Newtonian system, that is they satisfy equation (\ref{DIN}).

We say that $X$ is a \textbf{Hamilton-Jacobi vector field} associated to the Newtonian system $(Q,{\bf g},{\rm F})$.
\end{teor}
\begin{proof}

For any $p\in Q$, let $\gamma$ be the integral curve of $X$, $\dot{\gamma}=X\circ\gamma$, with initial condition $\gamma(0)=p$. Suppose that $\gamma$ satisfies (\ref{DIN}). Then at $t=0$ we have:
\begin{equation}
    \nabla_{\dot{\gamma}(0)}\dot{\gamma}(t)={\rm F}(\gamma(0))
\end{equation}
that is
\begin{equation}
    \nabla_{X(p)}X={\rm F}(p)\,.
\end{equation}
But the point $p\in Q$ is arbitrary, hence $\nabla_{X}X={\rm F}$ as we wanted.

On the other side, if $X$ satisfies $\nabla_{X}X={\rm F}$ and $\gamma$ is an integral curve of $X$, then
$$
 \nabla_{\dot{\gamma}(t)}\dot{\gamma}(t)= (\nabla_{X}X)(\gamma(t))={\rm F}(\gamma(t))\,,
$$
and the curve $\gamma$ satisfies equation (\ref{DIN}).
\end{proof}\qed

\bigskip
\noindent{\bf Comments}: 
\begin{enumerate}
  \item 
Notice that by this result, we obtain solutions of a second order differential equation, that is integral curves of a vector field in $\Tan Q$, as integral curves of vector fields in the manifold $Q$, any vector field solution to the equation $\nabla_X X={\rm F}$, that is first order differential equations. But, given a solution $X\in\mathfrak{X}(Q)$, we do not obtain from $X$ all the integral curves of the second order equation, we only obtain those with initial conditions  $(q,u)\in\Tan Q$ of the form $(q,X(q))$. 

In fact, for every solution $X$, we obtain a family of curves solution to the dynamical equation (\ref{DIN}): they are those solution curves contained in the submanifold of $\Tan Q$ defined by the graph of $X$ as a section of the natural projection $\tau_Q:\Tan Q\rightarrow Q$. It can be proved that this submanifold is invariant by the second order differential system associated to the dynamical equation. See \cite{CGMMR-06} for more details.
\item
If ${\rm F}=0$, we have the family of vector fields satisfying the equation $\nabla_X X=0$, the so called {\bf geodesic vector fields}. Their integral curves are geodesic curves for the Levi--Civita connection $\nabla$, they are solutions to the equation $\nabla_{\dot{\gamma}(t)}\dot{\gamma}(t)=0$. Some interesting properties of these vector field are in \cite{CMM--2021} and references therein.

\end{enumerate}

\subsection{Classical Hamilton-Jacobi equation in Lagrangian form}

Let $(Q,{\bf g},\omega=-\d V)$ be a conservative Newtonian system with force field
${\rm F}=-\mathrm{grad}\,V$. The Lagrangian function is $L=K-V$ and the total energy function is $E=K+V$.
Suppose the vector field $X$ is a Hamilton--Jacobi vector field of the system, $\nabla_X X={\rm F}$, then
$$
\mathcal{L}_X (E\circ X)=\mathcal{L}_X\left(\frac{1}{2}\textbf{g}(X,X)+V\right)=
$$
$$
=\textbf{g}(\nabla_X X,X)+\mathrm{d}V(X)= \textbf{g}(F,X)+\mathrm{d}V(X)=0
$$
which is a conservation of energy theorem for the Hamilton--Jacobi vector fields.

And following this last result we have:
\bigskip
\begin{teor}
(Hamilton--Jacobi equation)

Suppose that the vector field $X\in\mathfrak{X}(Q)$ satisfyes the condition 
$\mathrm{d}(\inn(X) \textbf{g})=0$, for example if $X$ has a potential function (see the comments below). Then the following conditions for $X$ are equivalent:
\begin{enumerate}
  \item $\nabla_X X=F=-\mathrm{grad}\,V$.
  \item $\mathrm{d}(E\circ X)=0$.
\end{enumerate}
\end{teor}
\begin{proof}

Let $Y$ be a vector field on $Q$, then:
\vspace{-2mm}
\begin{eqnarray*}
\mathrm{d}(E\circ X)(Y)&=&\mathcal{L}_Y(E\circ X)=\mathcal{L}_Y\left(\frac{1}{2}(\textbf{g}(X,X)+V\right)=\textbf{g}(\nabla_Y X,X)+\mathrm{d}V(Y)\\
 &=&\textbf{g}(\nabla_X X,Y)+\mathrm{d}V(Y)=\textbf{g}(\nabla_X X-F,Y)\,,
\end{eqnarray*}
where the fourth identity is consequence of
$$
\mathrm{d}(i_X \textbf{g})(Y,Z)=\textbf{g}(\nabla_Y X,Z)-\textbf{g}(\nabla_Z X,Y)
$$
directly obtained from the definition of the exterior differential and the symmetry of the connection $\nabla$.

But $Y$ is arbitrary, hence we have the equivalence.
\end{proof}
\qed

\bigskip
\noindent{\bf Comments}: 
\begin{enumerate}
\item Notice that the condition $\mathrm{d}(E\circ X)=0$ is equivalent to say that $E\circ X$ is a constant in every connected component of $Q$, that is
$$E(q^1,\ldots,q^n, X^1,\ldots,X^n)=\mathrm{constant}\,,$$
which is no more that the Hamilton--Jacobi equation but in Lagrangian form.
Its solutions are the vector fields $X\in\mathfrak{X}(Q)$ satisfying $\nabla_X X={\rm F}$ and $\mathrm{d}(\inn(X) \textbf{g})=0$. 

\item If we want to obtain the classical form of the Hamilton--Jacobi equation, that is the Hamiltonian form, we need to construct the Hamiltonian function, $H:\Tan^*Q\to\Real$, and change the vector field $X$ by the closed differential form $\alpha=\inn(X) \textbf{g}$. Then, locally, $\alpha=\d S$ and we obtain the classical equation. The closedness of the form is related to the special condition we have imposed to the vector field $X$ and it is equivalent to say that the image of the form $\alpha$ in $\Tan^*Q$ is a Lagrangian submanifold of $\Tan^*Q$ with its natural symplectic structure (see \cite{CGMMR-06}).

\item Condition $\mathrm{d}(i_X \textbf{g})=0$ for the vector field $X$ is equivalent to say that the image of the map $X:Q\to\mathrm{T}Q$ is a Lagrangian submanifold of the symplectic manifold $\mathrm{T}Q$ with the 2-Lagrangian form associated to $L=K-V$, which is a regular Lagrangian. It is related with the previous item, as $\alpha=i_X \textbf{g}$, and the Legendre transformation associated to the kinetic energy. Once again, see \cite{CGMMR-06} for more details.
\end{enumerate}

\subsection{Jacobi metric}

Consider a conservative mechanical system defined on the Riemannian manifold $(Q,{\bf g})$
with potential energy function $V \colon Q \to \Real$.
Recall that the energy is given by
$E(v_q) = \frac12 {\bf g}(v_q,v_q) + V(q)$
and it is a constant of the motion.

Suppose $E_0 > V(q)$ for all $q \in Q$.
We define the {\bf Jacobi metric} as
$$
{\bf g}_0 = (E_0-V) g .
$$
It is well known that the solutions $\gamma$ of the Newton equation
$\nabla_{\dot\gamma} \dot\gamma = -\mathrm{grad}\, V \circ \gamma$
with fixed energy $E_0$ are,
under convenient reparametrization,
the geodesic lines of~${\bf g}_0$. In \cite{God-69} there are interesting comments on this topic, and you can see a nice new approach of the same question in \cite{CMM--2021}.

\subsection{Applications}
\subsubsection{Euler equation for fluids as a Hamilton--Jacobi equation of a Newtonian system}

In this paragraph we interpret the Euler equation for stationary fluids as a Hamilton--Jacobi equation for a Newtonian system.

Let $U\subset\mathbf{R}^3$ an open set. The classical Euler equation for a time--depending vector field $X\in\mathfrak{X}(U)$, in Cartesian coordinates is given as
$$
\frac{\partial X}{\partial t}+<X,\nabla>X=-\nabla p
$$
where $p:U\times\mathbf{R}\to\mathbf{R}$ is a function, the pressure on the fluid.

If we consider the case where the fluid is incompressible, then we include  the condition $\mathrm{div}\,X=0$, that is the vector field is conservative. For a point $(x,t)$, the tangent vector $X(x,t)\in\Tan_x U$ is the velocity of the fluid particle in the point $x$ at the moment $t$.

In a Riemannian manifold $(M,\textbf{g})$, the corresponding Euler equation for a time--depending vector field $X\in\mathfrak{X}(M)$ is written as:
$$
\frac{\partial X}{\partial t}+\nabla_X X=-\mathrm{grad}\, p\, .
$$
The condition of being conservative for $X$ is $L_X \mathrm{d}V=0$, where $\mathrm{d}V$ is the volume element associated to $\textbf{g}$. The equivalence of both equations in direct in local coordinates.

For a stationary motion, that is constant in time, the last equation is:
$$
\nabla_X X=-\mathrm{grad}\, p
$$
that is, the Lagrangian Hamilton--Jacobi equation (\ref{HJL}) with ${\rm F}=-\mathrm{grad}\, p$, corresponding to the conservative Newtonian system $(M,{\bf g},\omega=-\d p)$.

\subsubsection{Hamilton--Jacobi equation and Schr\"odinger equation.
}
Given a Newtonian conservative system $(M,{\bf g},\omega=-\d V)$ and a vector field $X\in\mathfrak{X}(M)$ which is a gradient, $X=-\mathrm{grad}\,S$, consider the following three conditions:
\begin{enumerate}
  \item $X$ is a Hamilton--Jacobi vector field for the given system, that is
$\nabla_X X=-\mathrm{grad}\, V$ or equivalently $E\circ \mathrm{grad}\,S=E_o=\mathrm{constant}$.
\item $\mathrm{div} X=0$, hence $\mathrm{div}\, \mathrm{grad}\,S=\Delta S=0$, where $\Delta$ is the Laplacian operator for the Riemannian metric ${\bf g}$.
\item The function $\Psi= \exp(iS)$ satisfies the Schr\"odinger equation
$$
\left(-\frac{1}{2}\Delta+V\right)\Psi=E_o\Psi\, .
$$
\end{enumerate}
Then two of them imply the other one.

This is stated in \cite{ALOMU-2014} and is a nice relation between classical and quantum worlds, specially highlighted with this Riemannian approach to mechanics because all the elements included in the statement have a clear geometric meaning. For more explanations of these relations and interesting ideas see \cite{CMMGM-2016,MARMO-2009}. See also \cite{CHERO2001} for a Riemannian approach to quantum mechanics. We only include the first paper of a series of four.

\section{Other interesting topics}

In this section we give some references about topics where this Riemannian approach has given a modern revival of results and applications. We limite ourselves to give some information on some significant authors and recommend some of their publications and web pages as a source of information and update on these topics in the Riemannian approach. In any case, it is necessary to say that this is not a full list of groups or researchers which have contributed to the study of these problems with this approach.

\subsection{Symmetry and conserved quantities}

Symmetries of the system, associated conserved quantities and reduction by the action of symmetry groups are deep ideas in the study of mechanical systems from the very beginning. In our approach, the Riemannian structure gives a natural set of symmetries, the isometries of the metric and, infinitesimally, the Killing vector fields. For example, in the case that the configuration manifold is $\Real^3$, invariance by translations and rotations give rise to linear and angular momenta in the classic terminology.

We do not enter in historical references on  a so extended subject, which from a geometric viewpoint goes back to S. Lie in nineteenth century, with well recognized groups working on different topics and with a large amount of significant deep results. As an example we cite the work of G. Marmo and collaborators making geometry and symmetry fundamental tools to understand physical problems. See \cite{CIMM-2015} and the references therein as a book collecting their work along the years. Other well known books on the topic have been written by  J. Marsden and coauthors or A. Bloch and others authors. See, e.g. \cite{BLOCH2015} and the personal web pages of the authors, the web page of J. Marsden, specially maintained since he passed away in 2010, where most of his books and research are open accessible and are a source of information and results on this topic. These given references contain extended lists of names and topics related to the use of symmetry to understand problems in mechanics and other problems.

Hence we only give an example of a general result, a kind of {\bf Noether theorem} for the situation under study, that is mechanics with the Riemannian approach. 

\medskip
Suppose we have a Newtonian system $(M,{\bf g},{\rm F})$. For a vector field $X\in\vf(M)$, consider the function ${\bf I}_X=\widehat{\inn(X){\bf g}}:\Tan M\to\Real$  defined by ${\bf I}_X(v_q)={\bf g}(X_q,v_q)$. Then we have
\begin{enumerate}
\item If $X$ is a Killing vector field, that is $\mathcal{L}_X{\bf g}=0$, and $\sigma:I\subset\Real\to M$ is a geodesic line, that is $\nabla_{\dot\sigma}{\dot\sigma}=0$, then ${\bf I}_X$ is constant along $\dot\sigma$, that is 
\begin{equation}
\frac{\d}{\d t}({\bf I}_X\circ\dot\sigma)(t)=0\,.
\end{equation}
In other words, if $X$ is a Killing vector field, then ${\bf I}_X$ is a conserved quantity along the geodesic lines.
\item If $X$ is a Killing vector field satisfying ${\bf g}(X,{\rm F})=0$, and $\gamma:I\subset\Real\to M$ is a trajectory of the system $(M,{\bf g},{\rm F})$, that is $\nabla_{\dot\gamma}{\dot\gamma}={\rm F}\circ\gamma$, then ${\bf I}_X$ is constant along $\dot\gamma$, that is 
\begin{equation}
\frac{\d}{\d t}({\bf I}_X\circ\dot\gamma)(t)=0\,.
\end{equation}
In other words, if $X$ is a Killing vector field orthogonal to the force field ${\rm F}$, then ${\bf I}_X$ is a conserved quantity along the geodesic lines.
\end{enumerate}

Both results are consequence of the following equation
\begin{equation}
\frac{\d}{\d t}\left({\bf g}(X(\gamma(t),\dot\gamma(t))\right)=\nabla_{\dot\gamma}({\bf g}(X,\dot\gamma))={\bf g}(\nabla_{\dot\gamma}X,\dot\gamma)(t)+{\bf g}(X,\nabla_{\dot\gamma}\dot\gamma)(t)\, ,
\end{equation}
for every curve $\gamma$ and the well known property that $X$ is a Killing vector field if and only if 
${\bf g}(\nabla_YX,Y)=0$ for every $Y\in\vf(M)$ as can be seen from the following chain of identities:
\begin{eqnarray*}
  {\bf g}(\nabla_YX,Y)&=&{\bf g}(\nabla_XY+\mathcal{L}_YX,Y)={\bf g}(\nabla_XY,Y) -{\bf g}(\mathcal{L}_XY,Y)\\
  &=&\frac{1}{2}\mathcal{L}_X({\bf g}(Y,Y))-\left(\frac{1}{2}\mathcal{L}_X({\bf g}(Y,Y))-\frac{1}{2}(\mathcal{L}_X{\bf g})(Y,Y)\right)\\
  &=&\frac{1}{2}(\mathcal{L}_X{\bf g})(Y,Y)\, ,
\end{eqnarray*}
obtained using the properties of the Levi--Civita connection.

\medskip
\noindent{\bf Comment}: 
Which is the origin of the condition ${\bf g}(X,{\rm F})=0$ in item 2? Consider the case where the Newtonian system is conservative, then ${\rm F}=-{\rm grad}\, V$, hence the above orthogonality condition is equivalent to say that $\mathcal{L}_X V=0$, and this last condition, together with $X$ being Killing, implies that the Lagrangian $L=T-V$ is invariant by $X$, then the above result is a kind of generalization of the classical Noether theorem for Lagrangian systems.

\subsection{Ideas on control of mechanical systems}

Control of mechanical systems and robotics is a broad field of study both from the theoretical and the applied viewpoint. The geometric approach has a wide development at least from the eighties in the last century. Lagrangian, Hamiltonian and Riemannian approaches are used depending on the taste of the different authors and the closeness to the applications. In particular, Riemannian treatment is specially used in robotics. 

The control systems are modelled as a Newtonian system $(M,{\bf g},{\rm F})$ together with a set of control vector fields, or input forces, $F^1,\ldots,F^k$, we can modulate with some coefficients. The control equation is given as:
$$
\nabla_{\dot\gamma}{\dot\gamma}={\rm F}\circ\gamma+\sum_{i=1}^k u_i F^i\,.
$$
The coefficients $u_i$ are the input controls and, usually, are functions of time belonging to a specific set of functions and taking values in a specific domain $(u_1,\ldots,u_k)\in U\subset\Real^k$.

The full development of this ideas applied to control of mechanical systems can be seen in \cite{BULE2005}. It contains not only the state of the art and the work developed by the authors but  a detailed bibliography of other authors work. Applications in robotics and other fields, and recent developments, can be seen in the included references and in the web pages of the authors. In particular the importance of the so called {\bf symmetric product}, associated to the Levi--Civita connection, in the study of controllability of Newtonian control systems. The controllability of the above control mechanical system depends on properties of the symmetric algebra generated by the products $\ll F^i,F^j\gg=\nabla_{F^i}F^j+\nabla_{F^j}F^i$ of the control vector fields. Furthermore, it contains a complete description of the motion of a rigid body in this Riemannian approach.

In reference \cite{CORTES2002}, Ph. D. dissertation of the author, and in the subsequent research, there are several studies and applications in different specific  fields of  control of mechanical systems, in particular from the Riemannian point of view. In his personal web page there is a complete account of his developments and collaborators.

Optimal control of Newtonian systems is another topic where this Riemannian approach gives specific insight. Once again, the web pages of the above authors, in particular A. D. Lewis, are a good source of information on this topic. See also \cite{mbl-mcml-2009, CMZ-10}, and references therein, for an study of problems of optimal control in mechanical systems and the existence of special kind of solutions.

Another significant author working on this topic is  Arjan van der Schaft whose web page gives account of its developments with collaborators in the subject of control of mechanical and electromechanical systems. The use of different geometrical tools is permanent including the Riemannian approach.

\subsection{Other developments}

There are other interesting topics not included in this manuscript, for example forces depending on higher order derivatives, as elastic forces, used in mechanical models of continuous media. Sometimes these models include constraints depending also on second or third order derivatives. In \cite{Neimark-Fufaev}, a classical reference in nonholonomic systems, there are several examples of this type of systems. The approach in this book is classical and there will be interesting to develop some of the chapters in a more geometric terms, in particular in a Riemannian setting. See \cite{CILM-2004} for an interesting an pioonnering geometrical approach.

Moreover, not any reference is made in this manuscript to numerical methods of integration of dynamical equations, including those called ``geometric integrators" whose origin is in some geometric formulations of classical mechanics.

\section{Conclusions}
We have developed the analytical mechanics from a Riemannian geometry perspective including systems with constraints both holonomic and nonholonomic and non--autonomous systems. As specific topics with clear insight with our viewpoint, apart from the general theory, we include Hamilton--Jacobi vector fields and equation, the Jacobi metric and a specific Noether theorem for infinitesimal symmetries of the system. As applications we give the Euler equation for fluids as a Newtonian conservative system and one relation between specific solutions to the Hamilton--Jacobi equation an Schr\"odinger equation.

To finish we include some comments on control and optimal control of mechanical systems from a Riemannian perspective.

There are more topics, and authors working on them, where the Riemannian approach gives a special insight and sure they will continue giving new results an applications in different fields.

\section*{Acknowledgements}

\hspace{6
mm}Some ideas of this work come from a course given along the years in our Faculty of Mathematics and Statistics at the UPC for the students of the Mathematics degree. Both the course and this document would not have been possible without the permanent collaboration and friendship of my colleagues Narciso Rom\'an--Roy and Xavier Gr\`acia along the years. My deep thanks to them.

My former Ph.\,D. students Javier Yaniz--Fern\'andez and Mar\'ia Barbero--Li\~n\'an worked on control and optimal control of mechanical systems with this Riemannian approach. To both my reconnaissance for their dedication.

The author also
acknowledges the financial support from the Spanish Ministerio de Ciencia, 
Innovaci\'on y Universidades project PGC2018-098265-B-C33 and to
 the Secretary of University and Research of the Ministry of Business and 
Knowledge of the Catalan Government project 2017--SGR--932.

I also thank the referee for his careful reading of the manuscript and his comments that have allowed the author to improve some parts of the work.


\end{document}